\documentclass[12pt,twoside]{amsart}
\usepackage{amsmath, amsthm, amscd, amsfonts, amssymb, graphicx}
\usepackage{enumerate}
\usepackage[colorlinks=true,
linkcolor=blue,
urlcolor=cyan,
citecolor=red]{hyperref}
\usepackage{mathrsfs}
\newtheorem{theorem}{Theorem}[section]
\newtheorem{lemma}{Lemma}[section]
\newtheorem{remark}{Remark}[section]

\newtheorem{corollary}{Corollary}[section]

\numberwithin{equation}{section}

\begin{document}
	
\title{Buzano, Kre\u{\i}n and Cauchy-Schwarz inequalities}
\author{Mohammad Sababheh, Hamid Reza Moradi and Zahra Heydarbeygi}
\subjclass[2010]{47A30, 46C05, 15A60, 47A12.}
\keywords{Buzano inequality, Kre\u{\i}n inequality, Cauchy-Schwarz inequality, numerical radius.}

\begin{abstract}
The Cauchy-Schwarz, Buzano and Kre\u{\i}n inequalities are three inequalities about inner product.  The main goal of this article is to present  refinements of Buzano and Cauchy-Schwarz inequalities, and to present a new proof of a refined version of a Kre\u{\i}n-type inequality. Applications that include Buzano-type inequalities for certain operators, operator norm and numerical radius inequalities of Hilbert space operators will be presented.
\end{abstract}
\maketitle
%------------------------------------------------------------------------------------%
\pagestyle{myheadings}
\markboth{\centerline {M. Sababheh, H. R. Moradi \& Z. Heydarbeygi}}
{\centerline {Buzano, Kre\u{\i}n and Cauchy-Schwarz inequalities}}
\bigskip
\bigskip
%------------------------------------------------------------------------------------%
%------------------------------------------------------------------------------------%

\section{Introduction}

Let $\mathcal{H}$ be a given complex Hilbert space, with inner product $\left<\cdot,\cdot\right>$. The celebrated Cauchy-Schwarz inequality states that
\begin{equation}\label{eq_cs_1}
|\left<x,y\right>|\leq \|x\|\;\|y\|,
\end{equation}
for any vectors $x,y\in\mathcal{H}$. When $x$ and $y$ are non-zero vectors,  \eqref{eq_cs_1} implies  $0\leq \frac{|\left<x,y\right>|}{\|x\|\;\|y\|}\leq 1.$ This motivates defining the angle between the vectors $x,y$ by $\psi_{x,y}$ where
\begin{equation}\label{3}
\cos {{\psi }_{x,y}}=\frac{\left| \left\langle x,y \right\rangle  \right|}{\left\| x \right\|\left\| y \right\|}; 0\leq \psi_{x,y}\leq\frac{\pi}{2}.
\end{equation}
Another possible definition for the angle is $\varphi_{x,y}$ defined as 
	\[\cos {{\varphi }_{x,y}}=\frac{\operatorname{Re}\left\langle x,y \right\rangle }{\left\| x \right\|\left\| y \right\|}; 0\leq \varphi_{x,y}\leq \pi.\] 
We refer the reader to \cite{sch} for these definitions and some details.

In \cite{5}, Kre\u \i n obtained the following inequality for angles between two vectors
\begin{equation}\label{12}
{{\varphi }_{x,z}}\le {{\varphi }_{x,y}}+{{\varphi }_{y,z}},
\end{equation}
for any nonzero $x,y,z\in {{\mathbb{C}}^{n}}$.

In  \cite{4}, Lin   showed that the following triangle inequality 
\begin{equation}\label{13}
{{\psi }_{x,y}}\le {{\psi }_{x,z}}+{{\psi }_{z,y}},
\end{equation}
holds for any nonzero $x,y,z\in {{\mathbb{C}}^{n}}$. Lin's proof used the representation
	\[{{\psi }_{x,y}}=\underset{\alpha ,\beta \in \mathbb{C}\backslash \left\{ 0 \right\}}{\mathop{\inf }}\,{{\varphi }_{\alpha x,\beta y}}=\underset{\alpha \in \mathbb{C}\backslash \left\{ 0 \right\}}{\mathop{\inf }}\,{{\varphi }_{\alpha x,y}}=\underset{\beta \in \mathbb{C}\backslash \left\{ 0 \right\}}{\mathop{\inf }}\,{{\varphi }_{x,\beta y}}\] 
and inequality \eqref{12}.

Thus, both $\varphi_{x,y}$ and $\psi_{x,y}$ satisfy the triangle inequality. Our first target in this article is to present a new proof of \eqref{13}. This new proof will follow from some inner product inequalities, that we present while refining the celebrated Buzano inequality \cite{bu}, which states
\begin{equation}\label{4}
\left| \left\langle x,z \right\rangle  \right|\left| \left\langle y,z \right\rangle  \right|\le \frac{{{\left\| z \right\|}^{2}}}{2}\left( \left| \left\langle x,y \right\rangle  \right|+\left\| x \right\|\left\| y \right\| \right)
\end{equation}
for any $x,y,z\in \mathcal{H}$.
It is important to note that Buzano inequality gives a better bound than applying the Cauchy-Schwarz inequality twice on the left side. That is, \eqref{eq_cs_1} implies
\begin{align}\label{eq_cs_buz}
\left| \left\langle x,z \right\rangle  \right|\left| \left\langle y,z \right\rangle  \right|&\leq \|z\|^2\|x\|\;\|y\|.
\end{align}
At the same time, \eqref{eq_cs_1} implies 
$$\frac{{{\left\| z \right\|}^{2}}}{2}\left( \left| \left\langle x,y \right\rangle  \right|+\left\| x \right\|\left\| y \right\| \right)\leq \|z\|^2\|x\|\;\|y\|.$$
Consequently, \eqref{4} provides a refinement of \eqref{eq_cs_buz}.

After discussing the Buzano and Kre\u{\i}n inequalities, we present some new refinements of the Cauchy-Schwarz inequality \eqref{eq_cs_1} for positive contractive operators. 

As further and interesting applications of the obtained inequalities, we present some inequalities for the numerical radius and operator norm of Hilbert space operators. In this context, let $\mathbb{B}(\mathcal{H})$ denote the algebra of all bounded linear operators acting on a Hilbert space $\mathcal{H}.$  In $\mathbb{B}(\mathcal{H})$, an operator $A$ is said to be positive, and is denoted as $A\geq 0$, if $\left<Ax,x\right>\geq 0$ for all $x\in\mathcal{H}$. The partial ordering relation ``$\leq$'' is defined among self adjoint operators as
$$A\leq B\Leftrightarrow B-A\geq 0.$$
An operator $A\in\mathbb{B}(\mathcal{H})$ is said to be a positive contractive operator  if $A$ is  positive and $A\leq I$, where $I$ is the identity operator on $\mathcal{H}.$

We recall here that the operator norm and the numerical radius of an operator $T\in\mathbb{B}(\mathcal{H})$ are defined respectively by
$$\omega(T)=\sup_{\|x\|=1}|\left<Tx,x\right>|\;{\text{and}}\;\|T\|=\sup_{\|x\|=1}\|Tx\|.$$
It is well known that $\frac{1}{2}\|T\|\leq \omega(T)\leq \|T\|,$ for $T\in\mathbb{B}(\mathcal{H})$ (see e.g., \cite[Theorem 1.3-1]{gusta}). Our applications below include refinements of the second inequality above and some other consequences. Among many results, we retrieve the well known inequality \cite[Theorem 1]{drag2}:
$$\omega(ST)\leq\frac{1}{2}\|\;|S^*|^2+|T^2|\;\|, S,T\in\mathbb{B}(\mathcal{H}).$$
In this context, the notation $|X|$ will be used to denote $(X^*X)^{\frac{1}{2}}$, for $X\in\mathbb{B}(\mathcal{H}).$\\
Another interesting application of our results is another better bound for the numerical radius of the product of two operators. In particular, we deduce that when $B$ is positive, then (see  \cite[Corollary 2.6]{r2}) $$\omega(AB)\leq \frac{3}{2}\|B\|\omega(A), A\in\mathbb{B}(\mathcal{H}).$$ The significance of this inequality is due to the following inequalities
$$\omega(AB)\leq \|AB\|\leq \|B\|\;\|A\|\leq 2\|B\|\omega(A), A,B\in\mathbb{B}(\mathcal{H}).$$ Consequently, our new bound provides a considerable refinement of this latter bound. See Remark \ref{rem3} and Corollary \ref{cor3} below for the details.

Before proceeding to the main results, we present the following observation about projections and Buzano inequality, which can be considered as a new proof of the main result in \cite{drag4}.
\begin{remark}
Let $P$ be any orthogonal projection in $\mathbb{B}(\mathcal{H})$. Put $z=Px$ in \eqref{4}, and use the fact that ${{P}^{2}}=P={{P}^{*}}$,  for any projection, then for $x,y\in \mathcal{H}$,
\[\begin{aligned}
   {{\left\| Px \right\|}^{2}}\left| \left\langle y,Px \right\rangle  \right|&=\left| \left\langle Px,Px \right\rangle  \right|\left| \left\langle y,Px \right\rangle  \right| \\ 
 & =\left| \left\langle x,{{P}^{2}}x \right\rangle  \right|\left| \left\langle y,Px \right\rangle  \right| \\ 
 & =\left| \left\langle x,Px \right\rangle  \right|\left| \left\langle y,Px \right\rangle  \right| \\ 
 & \le \frac{{{\left\| Px \right\|}^{2}}}{2}\left( \left| \left\langle x,y \right\rangle  \right|+\left\| x \right\|\left\| y \right\| \right).  
\end{aligned}\]
Thus, we have shown that 
\[\left| \left\langle Px,y \right\rangle  \right|\le \frac{1}{2}\left( \left| \left\langle x,y \right\rangle  \right|+\left\| x \right\|\left\| y \right\| \right),\]
for any orthogonal projection $P$. It is interesting that positive contractive operators satisfy the Buzano inequality, as we show in Corollary \ref{10} below.
\end{remark}

\section{Buzano and Kre\u{\i}n inequalities}
We start this section with the following lemma, which can be obtained by \cite[Inequalities (1.5)-(1.6)]{r1}; however, for the reader's convenience, we add a proof. After that we present a new proof of the Kre\u{\i}n-Lin inequality \eqref{13}, with a refinement. 
\begin{lemma}\label{23}
For any $x,y,z\in \mathcal{H}$,
\[\begin{aligned}
  & \left| \left\langle x,z \right\rangle \left\langle z,y \right\rangle  \right| \\ 
 & \le \frac{1}{2}\left[ \left| \left\langle x,z \right\rangle \left\langle y,z \right\rangle  \right|+\left| \left\langle x,y \right\rangle  \right|{{\left\| z \right\|}^{2}}+\left| \left\langle x,y \right\rangle {{\left\| z \right\|}^{2}}-\left\langle x,z \right\rangle \left\langle z,y \right\rangle  \right| \right] \\ 
 & \le \frac{1}{2}\left[ \left| \left\langle x,z \right\rangle \left\langle y,z \right\rangle  \right|+\left| \left\langle x,y \right\rangle  \right|{{\left\| z \right\|}^{2}}+\sqrt{{{\left\| x \right\|}^{2}}{{\left\| z \right\|}^{2}}-{{\left| \left\langle x,z \right\rangle  \right|}^{2}}}\sqrt{{{\left\| y \right\|}^{2}}{{\left\| z \right\|}^{2}}-{{\left| \left\langle y,z \right\rangle  \right|}^{2}}} \right] \\ 
 & \le \frac{{{\left\| z \right\|}^{2}}}{2}\left( \left\| x \right\|\left\| y \right\|+\left| \left\langle x,y \right\rangle  \right| \right).  
\end{aligned}\]
\end{lemma}
\begin{proof}
We notice that if any of the vectors $x,y,z$ is the zero vector, then the result follows trivially. Let $x,y,z\in\mathcal{H}$ be any non-zero vectors. Then  
	\[\begin{aligned}
   \left| \left\langle x,y \right\rangle -\left\langle x,e \right\rangle \left\langle e,y \right\rangle  \right|&\le \sqrt{{{\left\| x \right\|}^{2}}-{{\left| \left\langle x,e \right\rangle  \right|}^{2}}}\sqrt{{{\left\| y \right\|}^{2}}-{{\left| \left\langle y,e \right\rangle  \right|}^{2}}} \\ 
 & \le \left\| x \right\|\left\| y \right\|-\left| \left\langle x,e \right\rangle \left\langle y,e \right\rangle  \right|, 
\end{aligned}\]
for any unit vector $e\in\mathcal{H}$. Replacing $e$ by $\frac{z}{\left\| z \right\|}$, 	
and multiplying by ${{\left\| z \right\|}^{2}}$, we infer that 
	
\begin{equation*}
\begin{aligned}
   \left| \left\langle x,z \right\rangle \left\langle z,y \right\rangle  \right|-\left| \left\langle x,y \right\rangle  \right|{{\left\| z \right\|}^{2}}&\le \left| \left\langle x,y \right\rangle {{\left\| z \right\|}^{2}}-\left\langle x,z \right\rangle \left\langle z,y \right\rangle  \right|.
\end{aligned}
\end{equation*}
Thus,
	\[\begin{aligned}
  & \left| \left\langle x,z \right\rangle \left\langle z,y \right\rangle  \right| \\ 
 & \le \frac{1}{2}\left[ \left| \left\langle x,z \right\rangle \left\langle y,z \right\rangle  \right|+\left| \left\langle x,y \right\rangle  \right|{{\left\| z \right\|}^{2}}+\left| \left\langle x,y \right\rangle {{\left\| z \right\|}^{2}}-\left\langle x,z \right\rangle \left\langle z,y \right\rangle  \right| \right] \\ 
 & \le \frac{1}{2}\left[ \left| \left\langle x,z \right\rangle \left\langle y,z \right\rangle  \right|+\left| \left\langle x,y \right\rangle  \right|{{\left\| z \right\|}^{2}}+\sqrt{{{\left\| x \right\|}^{2}}{{\left\| z \right\|}^{2}}-{{\left| \left\langle x,z \right\rangle  \right|}^{2}}}\sqrt{{{\left\| y \right\|}^{2}}{{\left\| z \right\|}^{2}}-{{\left| \left\langle y,z \right\rangle  \right|}^{2}}} \right] \\ 
 & \le \frac{{{\left\| z \right\|}^{2}}}{2}\left( \left\| x \right\|\left\| y \right\|+\left| \left\langle x,y \right\rangle  \right| \right).  
\end{aligned}\]
This completes the proof.
\end{proof}

To obtain the following result, we employ the strategy used in \cite[Inequality (1.9)]{r1}

\begin{corollary}
Let $x,y,z\in\mathcal{H}$ be any vectors. Then 
\begin{align*}
\psi_{x,y}&\leq \cos^{-1}\left( \cos {{\psi }_{x,y}}+\left| \cos {{\psi }_{x,y}}-\cos {{\psi }_{x,z}}\cos {{\psi }_{z,y}} \right|-\sin {{\psi }_{x,z}}\sin {{\psi }_{z,y}}\right)\\
&\leq \psi_{x,z}+\psi_{z,y}.
\end{align*}

\end{corollary}
\begin{proof}
It follows from the proof of Lemma \ref{23} that
\begin{equation}\label{2}
\begin{aligned}
  & \left| \left\langle x,z \right\rangle  \right|\left| \left\langle z,y \right\rangle  \right| \\  
 & \le \left| \left\langle x,y \right\rangle  \right|{{\left\| z \right\|}^{2}}+\sqrt{{{\left\| x \right\|}^{2}}{{\left\| z \right\|}^{2}}-{{\left| \left\langle x,z \right\rangle  \right|}^{2}}}\sqrt{{{\left\| y \right\|}^{2}}{{\left\| z \right\|}^{2}}-{{\left| \left\langle z,y \right\rangle  \right|}^{2}}}. 
\end{aligned}
\end{equation}
If we multiply \eqref{2}, by $0<\frac{1}{\left\| x \right\|\left\| y \right\|{{\left\| z \right\|}^{2}}}$, we get
\[\begin{aligned}
   \frac{\left| \left\langle x,z \right\rangle  \right|}{\left\| x \right\|\left\| z \right\|}\frac{\left| \left\langle z,y \right\rangle  \right|}{\left\| y \right\|\left\| z \right\|}&\le \frac{\left| \left\langle x,y \right\rangle  \right|}{\left\| x \right\|\left\| y \right\|}+\left| \frac{\left| \left\langle x,y \right\rangle  \right|}{\left\| x \right\|\left\| y \right\|}-\frac{\left| \left\langle x,z \right\rangle  \right|}{\left\| x \right\|\left\| z \right\|}\frac{\left| \left\langle z,y \right\rangle  \right|}{\left\| y \right\|\left\| z \right\|} \right| \\ 
 & \le \frac{\left| \left\langle x,y \right\rangle  \right|}{\left\| x \right\|\left\| y \right\|}+\sqrt{1-\frac{{{\left| \left\langle x,z \right\rangle  \right|}^{2}}}{{{\left\| x \right\|}^{2}}{{\left\| z \right\|}^{2}}}}\sqrt{1-\frac{{{\left| \left\langle z,y \right\rangle  \right|}^{2}}}{{{\left\| y \right\|}^{2}}{{\left\| z \right\|}^{2}}}}, 
\end{aligned}\]
which is equivalent to
\[\begin{aligned}
   \cos {{\psi }_{x,z}}\cos {{\psi }_{z,y}}&\le \cos {{\psi }_{x,y}}+\left| \cos {{\psi }_{x,y}}-\cos {{\psi }_{x,z}}\cos {{\psi }_{z,y}} \right| \\ 
 & \le \cos {{\psi }_{x,y}}+\sqrt{1-{{\cos }^{2}}{{\psi }_{x,z}}}\sqrt{1-{{\cos }^{2}}{{\psi }_{z,y}}} \\ 
 & =\cos {{\psi }_{x,y}}+\sin {{\psi }_{x,z}}\sin {{\psi }_{z,y}}  
\end{aligned}\]
by \eqref{3}. This implies
\[\begin{aligned}
  & \cos \left( {{\psi }_{x,z}}+{{\psi }_{z,y}} \right) \\ 
 & \le \cos {{\psi }_{x,y}}+\left| \cos {{\psi }_{x,y}}-\cos {{\psi }_{x,z}}\cos {{\psi }_{z,y}} \right|-\sin {{\psi }_{x,z}}\sin {{\psi }_{z,y}} \\ 
 & \le \cos {{\psi }_{x,y}}.
\end{aligned}\]
Now, since $\cos $ is a decreasing function on $\left[ 0,\pi  \right]$ and since $0\leq \psi_{x,z}+\psi_{z,y}\leq \pi$, the desired inequalities follow.

\end{proof}

Another refinement of the Cauchy-Schwarz inequality \eqref{eq_cs_1} can be stated as follows.
\begin{corollary}
For any $x,y,z\in \mathcal{H}$,
\[\begin{aligned}
   \left| \left\langle x,y \right\rangle  \right|{{\left\| z \right\|}^{2}}&\le \left| \left\langle x,z \right\rangle  \right|\left| \left\langle z,y \right\rangle  \right|+\sqrt{{{\left\| x \right\|}^{2}}{{\left\| z \right\|}^{2}}-{{\left| \left\langle x,z \right\rangle  \right|}^{2}}}\sqrt{{{\left\| y \right\|}^{2}}{{\left\| z \right\|}^{2}}-{{\left| \left\langle y,z \right\rangle  \right|}^{2}}} \\ 
 & \le {{\left\| z \right\|}^{2}}\left\| x \right\|\left\| y \right\|,  
\end{aligned}\]
holds. In particular,
\[\left| \left\langle x,y \right\rangle  \right|\le \left| \left\langle x,e \right\rangle  \right|\left| \left\langle e,y \right\rangle  \right|+\sqrt{{{\left\| x \right\|}^{2}}-{{\left| \left\langle x,e \right\rangle  \right|}^{2}}}\sqrt{{{\left\| y \right\|}^{2}}-{{\left| \left\langle y,e \right\rangle  \right|}^{2}}}\le \left\| x \right\|\left\| y \right\|,\]
where $e\in \mathcal{H}$ is a unit vector.
\end{corollary}

\section{Refinements of the Cauchy-Schwarz inequality via positive contractive operators with applications to the numerical radius}

The main results in this section include refinements of \eqref{eq_cs_1} via positive contractive operators. These refinements will lead to interesting applications including  numerical radius and operator norm inequalities. 

We first have the following simple observation.
\begin{lemma}
Let $A\in\mathbb{B}(\mathcal{H})$ be such that $0\leq A\leq 2I$ and let $x,y\in\mathcal{H}.$ Then
\begin{equation}\label{9}
\left| \left\langle x-Ax,y-Ay \right\rangle  \right|\leq  \left\| x \right\|\left\| y \right\|-\sqrt{\left\langle \left( 2A-{{A}^{2}} \right)x,x \right\rangle \left\langle \left( 2A-{{A}^{2}} \right)y,y \right\rangle }.
\end{equation}
\end{lemma}
\begin{proof}
We have
$$
\begin{aligned}
  & \left| \left\langle x-Ax,y-Ay \right\rangle  \right| \\ 
 & \le \left\| x-Ax \right\|\left\| y-Ay \right\| \\ 
 & ={{\left( {{\left\| x \right\|}^{2}}-\left\langle \left( 2A-{{A}^{2}} \right)x,x \right\rangle  \right)}^{\frac{1}{2}}}{{\left( {{\left\| y \right\|}^{2}}-\left\langle \left( 2A-{{A}^{2}} \right)y,y \right\rangle  \right)}^{\frac{1}{2}}} \\ 
 & \le \left\| x \right\|\left\| y \right\|-\sqrt{\left\langle \left( 2A-{{A}^{2}} \right)x,x \right\rangle \left\langle \left( 2A-{{A}^{2}} \right)y,y \right\rangle },
\end{aligned}
$$

where the second inequality follows from $\left( {{a}^{2}}-{{b}^{2}} \right)\left( {{c}^{2}}-{{d}^{2}} \right)\le {{\left( ac-bd \right)}^{2}}$; ($a,b,c,d\in \mathbb{R}^+$). This completes the proof.
\end{proof}
\begin{theorem}\label{thm_contr_1}
Let $A\in \mathbb{B}\left( \mathcal{H} \right)$ be a positive contractive  operator and let $x,y\in \mathcal{H}$. Then 
\begin{equation}\label{11}
0\le \sqrt{\left\langle \left( A-{{A}^{2}} \right)x,x \right\rangle \left\langle \left( A-{{A}^{2}} \right)y,y \right\rangle }-\left|\left\langle \left( A-{{A}^{2}} \right)x,y \right\rangle \right|\le \frac{\left\| x \right\|\left\| y \right\|-\left| \left\langle x,y \right\rangle  \right|}{4}.
\end{equation}
\end{theorem}
\begin{proof}
Observe that
\begin{equation}\label{8}
\begin{aligned}
   \left| \left\langle x-Ax,y-Ay \right\rangle  \right|&=\left| \left\langle x,y \right\rangle -\left\langle \left( 2A-A^2 \right)x,y \right\rangle  \right| \\ 
 & \ge \left| \left\langle x,y \right\rangle  \right|-\left| \left\langle \left( 2A-A^2 \right)x,y \right\rangle  \right|.  
\end{aligned}
\end{equation}

Combining inequalities \eqref{8} and \eqref{9} gives 
	\[\left| \left\langle x,y \right\rangle  \right|-\left| \left\langle \left( 2A-A^2 \right)x,y \right\rangle  \right|\le \left\| x \right\|\left\| y \right\|-\sqrt{\left\langle \left( 2A-{{A}^{2}} \right)x,x \right\rangle \left\langle \left( 2A-{{A}^{2}} \right)y,y \right\rangle }.\]
Whence,
	\[\sqrt{\left\langle \left( 2A-{{A}^{2}} \right)x,x \right\rangle \left\langle \left( 2A-{{A}^{2}} \right)y,y \right\rangle }-\left| \left\langle \left( 2A-A^2 \right)x,y \right\rangle  \right|\le \left\| x \right\|\left\| y \right\|-\left| \left\langle x,y \right\rangle  \right|.\] 
Replacing $A$ by $2A$, we get
	\[\sqrt{\left\langle \left( A-{{A}^{2}} \right)x,x \right\rangle \left\langle \left( A-{{A}^{2}} \right)y,y \right\rangle }-\left| \left\langle \left( A-{{A}^{2}} \right)x,y \right\rangle  \right|\le \frac{\left\| x \right\|\left\| y \right\|-\left| \left\langle x,y \right\rangle  \right|}{4},\] when $A\leq I$. This proves the second desired inequality. For the first inequality, since $A$ is positive contractive,  $A-{{A}^{2}}$ is positive. Then by the Cauchy-Schwarz inequality for positive operators, we get
\begin{equation*}
0\le \sqrt{\left\langle \left( A-{{A}^{2}} \right)x,x \right\rangle \left\langle \left( A-{{A}^{2}} \right)y,y \right\rangle }-\left|\left\langle \left( A-{{A}^{2}} \right)x,y \right\rangle\right|,
\end{equation*}
which proves the first inequality in \eqref{11}, and completes the proof of the
theorem.
\end{proof}

In fact, Theorem \ref{thm_contr_1} may be used to obtain the following easier form; as a refinement of the Cauchy-Schwarz inequality. It should be remarked that in \cite{r1}, this result was shown similarly for projections.

\begin{corollary}\label{6}
Let $A\in\mathbb{B}(\mathcal{H})$ be positive contractive. Then for $x,y\in\mathcal{H}$,
\begin{align*}
|\left<x,y\right>|+\sqrt{\left<Ax,x\right>\left<Ay,y\right>}-\left|\left<Ax,y\right>\right|\leq \|x\|\;\|y\|.
\end{align*}
In particular, if $A\in\mathbb{B}(\mathcal{H})$ is any nonzero positive  operator, then
\begin{align*}
|\left<x,y\right>|+\frac{1}{\|A\|}\left(\sqrt{\left<Ax,x\right>\left<Ay,y\right>}-\left|\left<Ax,y\right>\right|\right)\leq \|x\|\;\|y\|.
\end{align*}
\end{corollary}
\begin{proof}
In Theorem \ref{thm_contr_1}, we have shown that
\begin{equation}\label{11_1}
0\le \sqrt{\left\langle \left( A-{{A}^{2}} \right)x,x \right\rangle \left\langle \left( A-{{A}^{2}} \right)y,y \right\rangle }-\left|\left\langle \left( A-{{A}^{2}} \right)x,y \right\rangle\right| \le \frac{\left\| x \right\|\left\| y \right\|-\left| \left\langle x,y \right\rangle  \right|}{4},
\end{equation}
for the positive contractive operator $A\in\mathbb{B}(\mathcal{H})$ and $x,y\in \mathcal{H}$. Now, let $A\in\mathbb{B}(\mathcal{H})$ be such that $0\leq A\leq\frac{1}{4}I$, and define $$B=\frac{1}{2}\left(I+(I-4A)^{\frac{1}{2}}\right).$$ Since the mapping $t\mapsto t^{\frac{1}{2}}$ is operator monotone on $[0,\infty)$ \cite[Proposition V.1.8]{bhatia}, we have
\begin{align*}
0\leq A\leq \frac{1}{4}I&\Leftrightarrow 0\leq I-4A\leq I\\
&\Rightarrow 0\leq (I-4A)^{\frac{1}{2}}\leq I\\
& \Rightarrow 0\leq\frac{1}{2}\left(I+(I-4A)^{\frac{1}{2}}\right)\leq I.
\end{align*}
Then $0\leq B\leq I$. Therefore, \eqref{11_1} applies for $B$, and we have
$$0\le \sqrt{\left\langle \left( B-{{B}^{2}} \right)x,x \right\rangle \left\langle \left( B-{{B}^{2}} \right)y,y \right\rangle }-\left|\left\langle \left( B-{{B}^{2}} \right)x,y \right\rangle\right| \le \frac{\left\| x \right\|\left\| y \right\|-\left| \left\langle x,y \right\rangle  \right|}{4}.$$
But by the definition of $B$, we have $B-B^2=A$. This implies
$$0\le \sqrt{\left\langle  A x,x \right\rangle \left\langle  A y,y \right\rangle }-\left|\left\langle  A x,y \right\rangle\right| \le \frac{\left\| x \right\|\left\| y \right\|-\left| \left\langle x,y \right\rangle  \right|}{4},$$ when $0\leq A\leq \frac{1}{4}I.$ Now, if $A$ is an arbitrary positive contractive operator, replace $A$ by $\frac{1}{4}A$ in the above inequality. This implies the desired inequality and completes the proof.
\end{proof}

The following result is a Cauchy-Schwarz type inequality for positive contractive operators. Following this result, we explain how this extends \eqref{eq_cs_1}.
\begin{corollary}\label{10}
Let $A\in\mathbb{B}(\mathcal{H})$ be a positive contractive operator. Then for $x,y\in\mathcal{H}$,
\[\left| \left\langle Ax,y \right\rangle  \right|\le \frac{\left\| x \right\|\left\| y \right\|+\left| \left\langle x,y \right\rangle  \right|}{2}.\]
\end{corollary}
\begin{proof}
If $A\leq 2I$, it follows from \eqref{8} and \eqref{9} that
	\[\left| \left\langle x,y \right\rangle -\left\langle \left( 2A-{{A}^{2}} \right)x,y \right\rangle  \right|\le \left\| x \right\|\left\| y \right\|-\sqrt{\left\langle \left( 2A-{{A}^{2}} \right)x,x \right\rangle \left\langle \left( 2A-{{A}^{2}} \right)y,y \right\rangle }.\]
This implies
	\[\left| \left\langle x,y \right\rangle -\left\langle \left( 2A-{{A}^{2}} \right)x,y \right\rangle  \right|+\sqrt{\left\langle \left( 2A-{{A}^{2}} \right)x,x \right\rangle \left\langle \left( 2A-{{A}^{2}} \right)y,y \right\rangle }\le \left\| x \right\|\left\| y \right\|.\]
On the other hand, we have
	\[\left| \left\langle x,y \right\rangle -\left\langle \left( 2A-{{A}^{2}} \right)x,y \right\rangle  \right|+\left| \left\langle \left( 2A-{{A}^{2}} \right)x,y \right\rangle  \right|\le \left\| x \right\|\left\| y \right\|.\]
This implies
	\[-\left| \left\langle x,y \right\rangle  \right|+\left| \left\langle \left( 2A-{{A}^{2}} \right)x,y \right\rangle  \right|+\left| \left\langle \left( 2A-{{A}^{2}} \right)x,y \right\rangle  \right|\le \left\| x \right\|\left\| y \right\|,\] when $A\leq 2I$.
Thus,
	\[\left| \left\langle \left( 2A-{{A}^{2}} \right)x,y \right\rangle  \right|\le \frac{\left\| x \right\|\left\| y \right\|+\left| \left\langle x,y \right\rangle  \right|}{2},\] when $A\leq 2I$. 
Replacing $A$ by $2A$, we get
	\[\left| \left\langle \left( A-{{A}^{2}} \right)x,y \right\rangle  \right|\le \frac{\left\| x \right\|\left\| y \right\|+\left| \left\langle x,y \right\rangle  \right|}{8}.\] 
Applying the same procedure as in the proof of Corollary \ref{6}, we reach
	\[\left| \left\langle Ax,y \right\rangle  \right|\le \frac{\left\| x \right\|\left\| y \right\|+\left| \left\langle x,y \right\rangle  \right|}{2},\]
as desired.
\end{proof}
Notice that when $A=I$, Corollary \ref{10} implies \eqref{eq_cs_1}. Therefore, the above corollary provides an extension of \eqref{eq_cs_1}.

\begin{remark}\label{rem3}
Notice that when $A\in\mathbb{B}(\mathcal{H})$ is a given positive operator, replacing $A$ by $\frac{1}{\|A\|}A$ (when $A\not=0$) in Corollary \ref{10} implies
\begin{align*}
\left|\left<Ax,y\right>\right|\leq\frac{\|A\|}{2}\left(|\left<x,y\right>|+\|x\|\;\|y\|\right), x,y\in\mathcal{H}.
\end{align*}

However, this inequality is not true for an arbitrary nonzero $A\in\mathbb{B}(\mathcal{H}).$ This can be seen by taking the example:
$$A=\left[\begin{array}{cc}0&1\\0&0\end{array}\right], x=\left[\begin{array}{c}0\\1\end{array}\right], y=\left[\begin{array}{c}1\\0\end{array}\right].$$

 However, using the polar decomposition $A=U|A|$, one can see that the inequality
\begin{align*}
|\left<Ax,y\right>|&\leq \frac{\|A\|}{2}\left(|\left<Ux,y\right>|+\|x\|\;\|U^{*}y\|\right)\\
&\leq \frac{\|A\|}{2}\left(|\left<Ux,y\right>|+\|x\|\;\|y\|\right)
\end{align*}
 holds for $  x,y\in\mathcal{H}$ and  $A\in\mathbb{B}(\mathcal{H})$ with polar decomposition $A=U|A|.$ We notice here that $U$ is a partial isometry, and hence $\|U\|=\|U^*\|\leq 1.$ This justifies the second inequality above.

It is worth mentioning that, in the case when $B$ is a positive operator, the constant 4 in the inequality $\omega(AB)\leq 4\omega(A)\omega(B)$ can be reduced to ${3}/{2}\;$ as shown in the following way:
\[\begin{aligned}
   \omega \left( AB \right)&\le \frac{\left\| B \right\|}{2}\left( \omega \left( A \right)+\left\| A \right\| \right) \\ 
 & \le \frac{3}{2}\omega \left( A \right)\left\| B \right\| \\ 
 & =\frac{3}{2}\omega \left( A \right)\omega \left( B \right).  
\end{aligned}\]
In particular, if $B$ is positive contractive, then
\[\omega \left( AB \right)\le \frac{3}{2}\omega \left( A \right).\]
\end{remark}
As a conclusion of the above remark, and due to the importance of its finding, we summarize as follows. 
\begin{corollary}\label{cor3}
Let $A,B\in\mathbb{B}(\mathcal{H})$ be such that $B$ is positive. Then $$\omega(AB)\leq \frac{3}{2}\|B\|\omega(A).$$
\end{corollary}
It is worth mentioning that the above result has been proved in \cite[Corollary 2.6]{r2} using a different method.

 As another application of Corollary \ref{10}, we present the following numerical radius and operator norm applications. 

\begin{corollary}\label{corrr}
Let $A,S,T\in\mathbb{B}(\mathcal{H})$ be such that $A$ is positive contractive. Then
\begin{equation}\label{014}
\omega \left( SAT \right)\le \frac{1}{4}\left\| {{\left| T \right|}^{2}}+{{\left| {{S}^{*}} \right|}^{2}} \right\|+\frac{1}{2}\omega \left( {{ST}} \right).
\end{equation} 
Moreover,
\begin{equation}\label{15}
\left\| SAT \right\|\le \frac{\left\| T \right\|\left\| S \right\|+\left\| ST \right\|}{2}.
\end{equation}
In particular, 
\begin{align*}
\omega(ST)\leq \frac{1}{2}\left\| {{\left| T \right|}^{2}}+{{\left| {{S}^{*}} \right|}^{2}} \right\|.
\end{align*}
\end{corollary}
\begin{proof}
Replacing $x$ by $Tx$ and $y$ by ${{S}^{*}}x$ in  Corollary \ref{10}, we get
\[\left| \left\langle SATx,x \right\rangle  \right|\le \frac{\left\| Tx \right\|\left\| {{S}^{*}}x \right\|+\left| \left\langle Tx,{{S}^{*}}x \right\rangle  \right|}{2}.\]
Thus, by the arithmetic-geometric mean inequality, we obtain
	\[\begin{aligned}
   \left| \left\langle SATx,x \right\rangle  \right|&\le \frac{\left\| Tx \right\|\left\| {{S}^{*}}x \right\|+\left| \left\langle {{ST}}x,x \right\rangle  \right|}{2} \\ 
 & =\frac{\sqrt{\left\langle {{\left| T \right|}^{2}}x,x \right\rangle \left\langle {{\left| {{S}^{*}} \right|}^{2}}x,x \right\rangle }+\left| \left\langle {{ST}}x,x \right\rangle  \right|}{2} \\ 
 & \le \frac{1}{2}\left( \frac{1}{2}	\left({\left\langle {{\left| T \right|}^{2}}x,x \right\rangle +\left\langle {{\left| {{S}^{*}} \right|}^{2}}x,x \right\rangle }\right)+\left| \left\langle {{ST}}x,x \right\rangle  \right| \right) \\ 
 & =\frac{1}{2}\left( \frac{1}{2}\left\langle \left( {{\left| T \right|}^{2}}+{{\left| {{S}^{*}} \right|}^{2}} \right)x,x \right\rangle +\left| \left\langle {{ST}}x,x \right\rangle  \right| \right).  
\end{aligned}\]
Therefore,
	\[\left| \left\langle SATx,x \right\rangle  \right|\le \frac{1}{2}\left( \frac{1}{2}\left\langle \left( {{\left| T \right|}^{2}}+{{\left| {{S}^{*}} \right|}^{2}} \right)x,x \right\rangle +\left| \left\langle {{ST}}x,x \right\rangle  \right| \right).\]
Now, by taking supremum over all unit vector $x\in \mathcal{H}$, we get the inequality \eqref{014}.

To prove \eqref{15}, letting $x=Tx$ and $y={{S}^{*}}y$ in Corollary \ref{10}, we have
	\[\begin{aligned}
   \left| \left\langle SATx,y \right\rangle  \right|&\le \frac{\left\| Tx \right\|\left\| {{S}^{*}}y \right\|+\left| \left\langle Tx,{{S}^{*}}y \right\rangle  \right|}{2} \\ 
 & =\frac{\left\| Tx \right\|\left\| {{S}^{*}}y \right\|+\left| \left\langle STx,y \right\rangle  \right|}{2}.  
\end{aligned}\]
Now, the desired inequality \eqref{15} follows by taking supremum over $x,y\in \mathcal{H}$ with $\left\| x \right\|=\left\| y \right\|=1$.
\end{proof}
In dealing with numerical radius inequalities, we are interested in power inequalities. We refer the reader to \cite{drag2,drag3,kitt1,sattari} as a sample of references treating such inequalities. In the following result, we use Corollary \ref{corrr} to obtain a power inequality for the numerical radius.

\begin{corollary}
Let $A,S,T\in\mathbb{B}(\mathcal{H})$ be such that $A$ is positive contractive. Then
\begin{equation}\label{14}
\omega^r \left( SAT \right)\le \frac{1}{4}\left\| {{\left| T \right|}^{2r}}+{{\left| {{S}^{*}} \right|}^{2r}} \right\|+\frac{1}{2}\omega^r \left( {{ST}} \right),
\end{equation} 
for $r\geq 1$.
\end{corollary}
\begin{proof}
This follows from Corollary \ref{corrr} and the facts that $t\mapsto t^r, r\geq 1$ is a convex increasing function on $[0,\infty)$ and that
$$\left\|f\left(\frac{A+B}{2}\right)\right\|\leq \frac{1}{2}\|f(A)+f(B)\|,$$ for any increasing convex function $f:[0,\infty)\to [0,\infty)$ and positive operators $A,B,$ \cite[Corollary 2.2]{bourin}.
\end{proof}

Next, we use Corollary \ref{corrr} to obtain a refinement of the inequality $\omega(T)\leq \|T\|.$
\begin{corollary}
Let $T\in\mathbb{B}(\mathcal{H})$ be a given operator with the polar decomposition $T=U|T|.$ Then
\begin{align*}
\omega(T)&\leq \frac{1}{2}\left(\|T\|+\|T\|^{\frac{1}{2}}\omega\left(U|T|^{\frac{1}{2}}\right)\right)\\
&\leq \frac{1}{2}\left(\|T\|+\|T\|^{\frac{1}{2}}\left\|U|T|^{\frac{1}{2}}\right\|\right)\\
&\leq \frac{1}{2}\left(\|T\|+\|T\|^{\frac{1}{2}}\|U\| \|T\|^{\frac{1}{2}}\right)\\
&\leq\|T\|.
\end{align*}
\end{corollary}
\begin{proof}
We prove the first inequality, from which the other inequalities follow immediately. Let $T=U|T|$ be the polar decomposition of $T$. Then, for any vectors $x,y\in\mathcal{H}$,
\begin{equation}\label{needed_1}
\left|\left<Tx,y\right>\right|=\left|\left<U|T|x,y\right>\right|=\left|\left<|
T|^{\frac{1}{2}}x,|
T|^{\frac{1}{2}}U^*y\right>\right|.
\end{equation}
Let $A$ be any positive operator. Remark \ref{rem3} implies
\begin{align*}
\left|\left<Ax,y\right>\right|\leq\frac{\|A\|}{2}\left(\left|\left<x,y\right>\right|+\|x\|\;\|y\|\right).
\end{align*}
This together with \eqref{needed_1} imply
\begin{align*}
\left|\left<Tx,x\right>\right|&=\left|\left<|
T|^{\frac{1}{2}}x,|T|^{\frac{1}{2}}U^*x\right>\right|\\
&\leq\frac{\left\|\;|T|^{\frac{1}{2}}\right\|}{2}\left(\left|\left<x,|T|^{\frac{1}{2}}U^*x\right>\right|+ \|x\|\left\||T|^{\frac{1}{2}}U^*x\right\|          \right)\\
&=\frac{\|T\|^{\frac{1}{2}}}{2}\left(\left|\left<U|T|^{\frac{1}{2}}x,x\right>\right|+\|x\|^2\;  \|T\|^{\frac{1}{2}} \|U^*\|   \right).
\end{align*}
Noting that $\|U^*\|=\|U\|=1$ and taking the supremum over all unit vectors $x\in\mathcal{H}$, we obtain
$$\omega(T)\leq \frac{\|T\|^{\frac{1}{2}}}{2}\left(\omega\left(U|T|^{\frac{1}{2}}\right)+\|T\|^{\frac{1}{2}}\right).$$
This completes the proof of the first inequality, as desired.
\end{proof}

\section*{Acknowledgment} The authors would like to express there sincere gratitude to the anonymous reviewers, whose major comments have considerably improved the quality of the paper.

{\tiny \vskip 0.3 true cm }

{\tiny (M. Sababheh) Department of Basic Sciences, Princess Sumaya University For Technology, Al Jubaiha, Amman 11941, Jordan.}

{\tiny \textit{E-mail address:} sababheh@psut.edu.jo}

{\tiny \vskip 0.3 true cm }

{\tiny (H. R. Moradi) Department of Mathematics, Payame Noor University (PNU), P.O. Box 19395-4697, Tehran, Iran.}

{\tiny \textit{E-mail address:} hrmoradi@mshdiau.ac.ir }

{\tiny \vskip 0.3 true cm }

{\tiny (Z. Heydarbeygi) Department of Mathematics, Payame Noor University (PNU), P.O. Box 19395-4697, Tehran, Iran.}

{\tiny \textit{E-mail address:} zahraheydarbeygi525@gmail.com}

%-----------------------------------------------------------------------------
%-----------------------------------------------------------------------------
\end{document}